\documentclass[12pt,leqno]{amsart}
\usepackage{amssymb,enumerate}
\newcommand{\CC}{{\mathbb{C}}}
\newcommand{\FF}{{\mathbb{F}}}
\newcommand{\KK}{{\mathbb{K}}}

\newcommand{\QQ}{{\mathbb{Q}}}

\newcommand{\PP}{{\mathbb{P}}}

\newcommand{\AAA}{{\sf A}}
\newcommand{\SSS}{{\sf S}}

\newcommand{\Inn}{{\operatorname{Inn}}}

\newcommand{\Aut}{{\operatorname{Aut}}}
\newcommand{\Irr}{{\operatorname{Irr}}}
\newcommand{\St}{{\operatorname{St}}}
\newcommand{\Syl}{{\operatorname{Syl}}}

\newcommand{\ind}{{\operatorname{ind}}}
\newcommand{\Ker}{{\operatorname{Ker}}}

\newcommand{\diag}{{\operatorname{diag}\,}}
\newcommand{\GL}{{\operatorname{GL}}}
\newcommand{\SL}{{\operatorname{SL}}}

\newcommand{\PSL}{{\operatorname{L}}}
\newcommand{\PSU}{{\operatorname{U}}}

\newcommand{\SU}{{\operatorname{SU}}}
\newcommand{\Sp}{{\operatorname{Sp}}}

\newcommand{\SO}{{\operatorname{SO}}}

\newcommand{\PSp}{{\operatorname{S}}}
\newcommand{\OO}{{\operatorname{O}}}

\newcommand{\Sz} {{\operatorname{Sz}}}

\newcommand{\tw}[1]{{}^#1\!}
\renewcommand{\mod}{\bmod \,}

\let\eps=\epsilon
\newtheorem{thm}{Theorem}[section]
\newtheorem{theorem}[thm]{Theorem}
\newtheorem{lemma}[thm]{Lemma}
\newtheorem{cor}[thm]{Corollary}
\newtheorem{prop}[thm]{Proposition}
\newtheorem{conj}[thm]{Conjecture} 

\newtheorem{example}[thm]{Example}
\theoremstyle{definition}

\theoremstyle{remark}

\begin{document}

\title[Lifting in Frattini Covers and A Characterization of Solvable Groups]
{Lifting in Frattini Covers and A Characterization of Finite Solvable Groups}

%%\date{\today}

\author{Robert M. Guralnick}
\address{Department of Mathematics, University of Southern California, Los Angeles, CA 90089-2532, USA}
\email{guralnic@usc.edu} 
\author{Pham Huu Tiep}
\address{Department of Mathematics\\
    University of Arizona\\
    Tucson, AZ 85721\\ 
    U. S. A.} 
\email{tiep@math.arizona.edu}

\thanks{The authors gratefully acknowledge the support of the NSF 
(grants DMS-1001962 and DMS-0901241). They are also grateful to 
M. Fried for helpful comments on the topic of the paper.}


 
\begin{abstract}  
We prove a lifting theorem for odd Frattini covers
of finite groups.  Using this, 
we characterize finite solvable groups as those finite groups which 
do not contain nontrivial elements $x_i$, $i=1,2,3$, with $x_1x_2x_3=1$  and
$x_i$ a $p_i$-element for distinct primes $p_i$.
 \end{abstract}

\maketitle

%%\pagestyle{myheadings}
%%\markboth{for personal use only}{preliminary version}
%%\markboth{}{}

%%%%%%%%%%%%%%%%%%%%%%%%%%%%%%%%%%% 
\section{Introduction} \label{sec:intro}

Let $G$ be a finite group.  Suppose that $X$ is a finite group such that 
$X/F = G$. We say that $X$ is a {\it Frattini cover of $G$} if $F$ is contained
in the Frattini subgroup $\Phi(X)$ of $X$.

Frattini covers have been studied considerably with respect
to coverings of curves (and infinite towers of coverings of curves).

Fried \cite{intro, fried}  introduced the  modular tower problem and has
made several interesting conjectures regarding them which generalize
the fact that for an elliptic curve defined over a number field,  the torsion
subgroup defined over the number field is bounded.    See \cite{fried, BF, FK, D}
for much more about this problem and its motivation and interesting
examples.   Theorems \ref{main1} and \ref{main1gen} have some interesting consequences
for Hurwitz spaces. In particular, it  implies the existence of certain
Hurwitz spaces for  certain Frattini covers related to the modular tower program. 

Our first result gives a lifting critetion in Frattini covers. 
%depends critically on a result of Isaacs on
%what he calls {\it character fives} \cite{I1}.   
%% took out hypothesis of no abelian p quotient as it already follows

\begin{theorem}  \label{main1}  
Let $p$ be an odd prime. Let $X$ be a Frattini cover of $G = X/F$
with $F$ a $p$-group. Assume that $p$ does not divide the order of
the Schur multiplier of $G$.   Let $g_1, \ldots, g_r \in G$  satisfy

{\rm (i)} $G = \langle g_1, \ldots, g_r \rangle$,

{\rm (ii)}  $g_1 \cdots g_r =1$, and

{\rm (iii)} the order of each $g_i$ is coprime to $p$.\\ 
Then, for any $f \in F$, there exist $x_i \in X$, with $x_iF = g_i$, $|x_i| = |g_i|$,
such that $x_1\cdots x_r = f$.
\end{theorem}

In particular, we can take $f=1$.   This shows that the lifting invariant
defined by Fried vanishes in this setting.  Furthermore, 
by a standard compactness argument, we can
take $X$ to be the $p$-universal Frattini cover (see \cite{fried}). 
More generally, if we drop the assumption that $p$ does not divide
the order of the Schur multiplier, we see that the only obstruction
to lifting is lifting to the maximal central Frattini cover
(and this is a true obstruction -- see the examples in Section \ref{lifting}
and Section \ref{sec:psolvable}). 
See Theorem \ref{main1gen} for a more general result which does not
assume the condition on the Schur multiplier. 

Of course, we can replace $p$ by any set $\pi$ of odd primes   and let  $F$ be a (necessarily nilpotent) 
$\pi$-group (see Theorem \ref{main1gen}).
  We also construct examples to show that
Theorem \ref{main1} fails if $p=2$, see Proposition \ref{p=2}.

Using this result and  the Thompson classification of the finite simple
groups in which every proper subgroup is solvable \cite{jgt}, we obtain
the following characterization of finite solvable groups (Barry \cite{barry}
asked whether this was true).

\begin{theorem} \label{main2}
Let $G$ be a finite group.  Then $G$ is solvable if and only if
$x_1x_2x_3 \neq 1$  for all nontrivial $p_i$-elements $x_i$ of $G$ for distinct
primes $p_i$, $i = 1,2,3$.
\end{theorem}

The forward implication is  trivial. The reverse implication crucially depends on 
results of Thompson \cite{jgt}, and Proposition \ref{split} which  
essentially follows from a result of Isaacs \cite{I1}.  
Thompson proved this result  if one considers all triples of nontrivial elements
of corpime order using his result on simple groups with all proper subgroups solvable. 
Kaplan and Levy \cite{kap} (see also  \cite{GL}) proved a variant of the previous result  -- in their result,
$x_1$ is a $2$-element, $x_2$ is a $p$-element for some odd prime $p$ and
$x_3$ is any  nontrivial element whose order is coprime to $2p$.   We actually
prove a somewhat stronger result by showing that in 
Theorem \ref{main2} it suffices
to assume that $p_1=2$ and $p_2 \in \{3,5\}$  (see Theorem \ref{main3}). 
An even stronger result, Theorem \ref{main4} characterizing $p$-solvable groups, is also obtained (but using
the full classification of finite simple groups).

There have been many characterizations of finite solvable groups.  We mention a few:

\begin{enumerate}  
\item Every $2$-generated subgroup is solvable \cite{jgt} (see also \cite{flav});
\item  Every pair of conjugate elements generate a solvable group \cite{guest};
\item  The proportion of pairs of elements which generate a solvable group
is greater than $11/30$ \cite{gurwil}; 
\item  If $x, y \in G$, then $x, y^g$ generate a solvable group for some $g \in G$
\cite{dghp}.
\end{enumerate}  

See \cite{dghp, GL, G2} for more characterizations and other references.

We obtain another characterization of solvable groups  
by combining our Theorem \ref{main3} with the proof of \cite[Theorem 2]{barry}:

\begin{cor}  Let $G$ be a finite group.  The following are equivalent:

{\rm (i)}   $G$ is solvable;

{\rm (ii)}  For all distinct primes $p_i$ and for all Sylow $p_i$-subgroups $P_i$
of $G$ with $1 \le i \le 3$, $|P_1P_2P_3|=|P_1|  |P_2| |P_3|$.

{\rm (iii)} For all distinct primes $p_i$ and for all Sylow $p_i$-subgroups $P_i$
of $G$, where $1 \le i \le 3$, $p_1=2$, and $p_2 \in \{3,5\}$, 
$|P_1P_2P_3|=|P_1| |P_2| |P_3|$.
\end{cor}

Note that  Theorem \ref{main2} (and Theorem \ref{main3}) depends on Thompson's results 
but not on the full classification of finite simple groups. In fact, the only results 
in this paper which depend on the full classification are the ones  in \S5.    

In the next section, we prove Theorem \ref{main1} and the more general
Theorem \ref{main1gen}.  In section 3 we
prove Theorem \ref{main3} (which includes Theorem \ref{main2}).   In 
section 4, we elaborate on the fact that one cannot just take elements of prime order
in Theorem \ref{main2}. In section 5, using the full classification of finite
simple groups,  we 
characterize the $p$-solvable finite groups (for $p \geq 3$):

\begin{thm} \label{main4}  
Let $p$ be an odd prime and $G$ be a finite group. Then $G$ is $p$-solvable if and only
if  $G$ does not admit a triple of nontrivial elements $x,y,z$ with $xyz = 1$, $x$ a 
$2$-element, $y$ a $p$-element, and $z$ a $q$-element for any odd prime $q \neq  p$.
\end{thm}

   In section \ref{even}, we give examples to show that
Theorem \ref{main1} fails for $p=2$.   In the final section,  
using Proposition \ref{split}, 
we give a short proof of
a theorem of Feit and Tits \cite{FT} about the minimal dimension of a representation
of a group which has a section isomorphic to a given simple group. 

We use the notation of \cite{Atlas} for various finite simple groups
(in particular, $\PSL_n(q)$, $\PSU_n(q)$, $\PSp_{2n}(q)$, and
$\OO^{\pm}_{n}(q)$ stands for ${\mathrm {PSL}}_n(q)$, 
${\mathrm {PSU}}_n(q)$, ${\mathrm {PSp}}_{2n}(q)$, and $P\Omega^{\pm}_{n}(q)$,
respectively).

\section{Lifting}  \label{lifting} 

The following statement is a key ingredient in our further considerations. It is 
essentially a consequence of a result of Isaacs \cite[Theorem 9.1]{I1}, and it is 
probably also known to Dade. For the sake of completeness, we give an independent
proof of it.  

\begin{prop}\label{split}
Let $G$ be a finite group and $N$ a non-abelian normal $p$-subgroup of $G$ for a prime 
$p$. Assume that $N$ is minimal among noncentral normal subgroups of $G$. 

{\rm (i)} If $\varphi \in \Irr(Z(N))$ is nontrivial on $[N,N]$, 
then $\varphi$ is fully ramified with respect to $N/Z(N)$, i.e.  
$\varphi^{N} = e\theta$ for some $\theta \in \Irr(N)$ and $e^2 = |N/Z(N)|$.

{\rm (ii)} Assume $p > 2$. Then $G/Z(N)$ splits over $N/Z(N)$. In particular,
$N$ is not contained in $\Phi(G)$.
\end{prop}

\begin{proof}
1) Observe that since $N$ is non-abelian, $N$ is noncentral in $G$. 
Since $1 < Z(N) < N$, the minimality of $N$ implies that $Z(N) \leq Z(G)$, 
and so $Z(N) = Z(G) \cap N$. Similarly, $[N,N] \leq Z(G)$, whence 
$[N,N] \leq Z(N)$ and $N$ is nilpotent of class $2$. It follows that 
for all $x,y,z \in N$, $[x,z][y,z] = [xy,z]$. In particular,
for all $x \in N$, the set $[x,N] := \{ [x,n] \mid n \in N\}$ is a subgroup of
$[N,N]$. Next we claim that $[x,N] = [N,N]$ for all $x \in N \setminus Z(N)$. 
Assume the contrary: $X := [x,N] < [N,N]$. Since $N$ has class $2$, 
$Y := \{ y \in N \mid [y,N] \leq X \}$ is a subgroup of $N$ containing 
$x$ and $Z(N)$. Next, $X \leq Z(G)$ is normal in $G$, hence $Y \lhd G$. 
If $Y = N$, then $[N,N] \leq X$, a contradiction. Therefore $Y < N$, and
so $Y \leq Z(G) \cap N = Z(N)$ by minimality of $N$, again a contradiction.

2) Now we prove (i). By the assumption,  $\varphi(z) \neq 1$ for some $z \in [N,N]$. 
%Since $Z(N) \geq [N,N]$ is abelian, $\lambda$ extends to
%$\varphi \in \Irr(Z(N))$, which is $G$-invariant as $Z(N) \leq Z(G)$. 
Let $\theta \in \Irr(N)$ be any irreducible constituent of $\varphi^N$. 
We claim that $\theta(x) = 0$ for all $x \in N \setminus Z(N)$. 
(Indeed, let $\Theta$ be a complex representation affording the character $\theta$. 
By the previous paragraph, there is some $n \in N$ such that $z = x^{-1}n^{-1}xn$. 
Since $\Theta(z) = \varphi(z)I$, it follows that 
$\Theta(n)^{-1}\Theta(x)\Theta(n) = \varphi(z)\Theta(x)$. Taking traces, we see
that $\theta(x)(\varphi(z)-1) = 0$ and so $\theta(x) = 0$ by the choice of $z$.)
It is well known that in this case $\varphi$ is fully ramified with respect to 
$N/Z(N)$ (cf. \cite[Lemma 2.6]{I1}).

3) From now on, we assume that $p > 2$.
Let $M \geq Z(N)$ be chosen such that 
$M/Z(N) = \Omega_1(N/Z(N))$. Then $Z(G) \cap N = Z(N) < M \leq N$ and $M \lhd G$. 
The minimality of $N$ again implies that $M = N$, i.e. $N/Z(N)$ is elementary abelian.
Now for any $x,y \in N$ we have $x^p \in Z(N)$ and so 
$[x,y]^p = [x^p,y] = 1$. Since $[N,N]$ is abelian, it follows that $[N,N]$ is 
also elementary abelian. As $p > 2$, this also implies that 
$(xy)^p = x^py^p$ for all $x,y \in N$.

4) Now we pick any nontrivial $\lambda \in \Irr([N,N])$. 
Since $Z(N) \geq [N,N]$ is abelian, $\lambda$ extends to
$\varphi \in \Irr(Z(N))$, which is $G$-invariant because $Z(N) \leq Z(G)$. 
Let $K := \Ker(\varphi)$ and let $P := N/K$; in particular,
$K \leq Z(G)$. We will now show that 
$P \cong p^{1+2n}_{+}$, an extraspecial $p$-group of exponent $p$ of order
$p^{1+2n}$ for some $n \geq 1$. 

First, $K \cap [N,N] = \Ker(\lambda)$ has index $p$ in $[N,N]$ as $[N,N]$ is 
elementary abelian. Hence $[P,P] = K[N,N]/K$ is cyclic of order $p$;
in particular, $P$ is non-abelian. 
Choose $N_1 \geq Z(N)$ such that $N_1/K = Z(N/K)$. Then $N_1 < N$ and 
$N_1 \lhd G$. The minimality of $N$ implies that $N_1 = Z(N)$. Also,
$\varphi$ is a faithful linear character of $Z(N)/K = N_1/K$. We have shown
that $Z(P) = Z(N)/K$ is cyclic.

Next we claim that $P$ contains noncentral elements of order $p$. Indeed,
fix $x \in P$ such that $|x| = \exp(P) = p^s$. 
Since $P$ is non-abelian, we can find $y \notin \langle x,Z(P) \rangle$. 
Then $|y| = p^t$ with $1 \leq t \leq s$ by the choice of $x$. Also, since
$P/Z(P) \cong N/Z(N)$ is elementary abelian, $x^p,y^p \in Z(P)$. Now 
$x^p$, respectively $y^p$, is an element in the cyclic group $Z(P)$, of 
order $p^{s-1}$, respectively $p^{t-1}$, and $t \leq s$. It follows that there is 
some integer $k$ such that $y^p = x^{kp}$. As shown in 3), we now have 
$(x^{-k}y)^p = x^{-kp}y^p = 1$ and $x^{-k}y \notin Z(P)$, i.e. $x^{-k}y$ is a noncentral
element of order $p$ in $P$, as desired.

Let $N_2 := \{ x \in N \mid x^p \in K\}$. If $x, y \in N_2$, then by 3) we have 
$(xy)^p = x^py^p \in K$ and so $xy \in N_2$. Hence $N_2 \lhd G$. By the previous 
claim, $N_2$ is not contained in $Z(N) = Z(G) \cap N$. The minimality of $N$ now 
implies that $N_2 = N$. We have shown that $\exp(P) = p$. Since $Z(P)$ is cyclic, we 
also have $Z(P) \cong C_p \cong [P,P]$, and so $Z(P) = [P,P]$. But 
$P/Z(P) \cong N/Z(N)$ is elementary abelian, hence $\Phi(P) = Z(P)$. 
Thus $P$ is an extraspecial $p$-group of exponent $p$, as stated.

5) It is well known that $\Aut_1(P)$, the group of all automorphisms of
$P$ which act trivially on $Z(P)$, is a semidirect product $IS$, where 
$I = \Inn(P) \cong P/Z(P)$ is of order $p^{2n}$ and $S \cong \Sp_{2n}(p)$. Now set
$C := C_G(N/K)$ so that $H := G/C$ embeds in $\Aut_1(P)$. Since $[N,N] \not\leq K$,
we have $Z(N) \leq N \cap C < N$, and so $N \cap C = Z(N)$ by the minimality of $N$. 
Thus $NC/C \cong N/Z(N) \cong P/Z(P) \cong I$. Now we can certainly write 
$H$ as a semidirect product of $NC/C$ and $H \cap S$. Let $U \geq C$ be such
that $U/C = H \cap S$; in particular $U \cap NC = C$ since $S \cap I = 1$. 
Then $G = (NC)U = UN$, and $U \cap N = (U \cap NC) \cap N = C \cap N = Z(N)$, i.e. 
$G/Z(N)$ splits over $N/Z(N)$.

Suppose now that $N \leq \Phi(G)$. Then $G =\Phi(G)U$ implies that 
$G = U$, contradicting the fact that $U \cap N = Z(N)$.
\end{proof}  
  
If we consider groups of even order, the previous result fails.  For example,
there is a non-split extension of an extraspecial $2$-group of order
$2^{1+2a}$ by an orthogonal group $\OO_{2a}(2)$ (if $a \geq 5$).

\begin{lemma}  \label{abelian}  
Let $p$ be a prime and let $G$ be a finite group generated by 
elements $g_1, \ldots, g_r$.  
Assume that either $G = O^p(G)$, or all $g_i$ are $p'$-elements.  
%%%Certainly the second condition implies the first, but we keep this 
%%%redundant formulation for convenience.
Let $N$ be an abelian
$p$-subgroup of $G$ that is a minimal noncentral normal subgroup of $G$.  
Let $C_i:=[g_i,N]$. Then
$(g_1C_1)(g_2C_2) \cdots (g_rC_r) = (g_1 \ldots g_r)N$. Furthermore, 
every element of $g_iC_i$ is $N$-conjugate to $g_i$ and is contained in
the coset $g_iN$.
\end{lemma}

\begin{proof}  
Note that, since
$N$ is normal and abelian,  $C_i=\{[g_i,n] \mid  n \in N\}$ and it is a subgroup
of $N$.  Clearly,  any element of $g_iC_i$ is a conjugate of $g_i$
(by an element of $N$).  It is straightforward to see that
$$(g_1C_1)(g_2C_2) \cdots (g_rC_r)   =  (g_1 \cdots g_r) D_1 \cdots D_r,$$
where $D_r = C_r,  D_{r-1}=C_{r-1}^{g_r}, \ldots ,   D_1 = C_1^{g_2 \cdots g_r}$.

Suppose that $[G,N] \ne N$.  Then 
$[G,N]$ is central in $G$ by the minimality of $N$.  
If $h$ is any $p'$-element of $G$, this implies
that $[h^{p^a},n]=[h,n]^{p^a} = 1$ for all $n \in N$, if $\exp(N) = p^a$. 
It follows that all $p'$-elements of $G$ centralize $N$.
By the hypothesis, $G$ is generated by $p'$-elements, whence $G$ centralizes $N$,
a contradiction. Hence $[G,N] = N$.  
 
Set $M:= D_1 \cdots D_r$.  Note that $[g_r,N] \le M$ and so $g_r$ acts trivially
on $N/M$. Also, $[g_r,M] \leq [g_r,N] = D_r \leq M$, and so $M^{g_r} = M$. 
Hence $C_{r-1} = D_{r-1}^{g_r^{-1}} \leq M^{g_r^{-1}} = M$, whence $g_{r-1}$ acts 
trivially on $N/M$. Continuing in this manner, we see that each $g_i$ acts 
trivially on $N/M$. Since $G = \langle g_1, \ldots ,g_r \rangle$,
it follows that $M \ge [G,N]=N$. The last statement is obvious. 
\end{proof}

We can now prove Theorem \ref{main1}.
 
\begin{proof}   
Let $G = X/F$ as in the statement. We induct on $|F|$.  If
$F=1$, there is nothing to prove.  So assume that this is not the case. 
Observe that $O^{p}(X) = X$ and so $p$ does not divide $|X/[X,X]|$. Indeed,
$X/O^{p}(X)F \geq O^{p}(G) = G$, whence $O^{p}(X)F = X$ and so 
$O^{p}(X) = X$ since $F = \Phi(X)$.

First assume that $F \leq Z(G)$. Since $p$ is coprime to $|X/[X,X]|$, 
$F \leq [X,X]$. Thus $X$ is a central extension of $G$ with kernel $F$ contained
in $[X,X]$. Hence $F \hookrightarrow {\mathrm {Mult}}(G)$ by \cite[Corollary 11.20]{I2}.
In particular, $p$ divides $|{\mathrm {Mult}}(G)|$, a contradiction.  
 
Thus $F$ is not central and so we can take 
$N \leq F$ to be a minimal normal noncentral subgroup of $X$. 
By Proposition \ref{split}(ii), $N$ is abelian.  
 
By the induction hypothesis applied to $X/N$, 
we can choose $x_i \in X$ of order coprime to $p$ such 
that $g_i = x_iF$ and $x_1 \ldots x_r = fn$ for some $n \in N$. Note that 
$\langle x_1, \ldots x_n \rangle = X$. (Otherwise 
$Y := \langle x_1, \ldots x_n \rangle = X$ is contained in a maximal subgroup
$M$ of $X$. But $YF = X$ and $F = \Phi(X) \leq M$, so $M = X$, a contradiction.)
 By Lemma \ref{abelian} applied to $X$, 
there exist $y_i \in x_iN$ with $y_i$ conjugate
to $x_i$ such that  $y_1 \cdots y_r = (x_1 \cdots x_r) n^{-1}=f$.
\end{proof} 

Note that Theorem \ref{main1} can be viewed as a result about
branched coverings of Riemann surfaces.   Suppose that
$f~:~Y \rightarrow \PP^1$ is a Galois branched covering of
a Riemann surface $Y$ with Galois group $G$ and all ramification 
coprime to a given prime $p > 2$.   Let $G=X/F$ be a Frattini cover
with $F$ a $p$-group and assume that $F=[X,F]$.   Then
there exists an unramified $F$-cover $Z \rightarrow Y$ with
$Z \rightarrow \PP^1$ Galois.   
 
Bailey and Fried \cite{BF} have shown that if $p=2$ then lifting to a
central Frattini extension is not always possible.  
 
More generally, we cannot remove the condition that the prime $p$ is coprime
to   $|{\mathrm {Mult}}(G)|$ in Theorem \ref{main1}.  
We give families of such examples for any prime $p$.
 
\begin{example}  
{\em Let $G = \PSL_2(q)$ with $q > 3$ an odd prime power and $q$ not 
a Fermat prime. Let $C$ be a conjugacy class of elements of 
odd prime power order dividing $(q-1)/2$ such that
 $G = \langle x, y \rangle$ with $x, y \in C$ and $z^{-1} = xy \in C^2$
 (it is a straightforward computation to see that these exist -- see 
\cite[Lemma 3.14]{GM}
 or \cite{M}).   Let
 $X = \SL_2(q)$, and let $D$ be the conjugacy of elements 
 of odd order which is the lift of $C$ to $X$.   Then we cannot find
 $u,v \in D$ and $w \in D^{-2}$ with $uvw = 1$ and 
$X = \langle u,v,w \rangle$.
(Indeed, there is some $\alpha \in \FF_q^{\times}$ such that 
$\alpha u$, $\alpha v$, and $\alpha^{-2}w$ all have a one-dimensional 
fixed point subspace on the natural $X$-module $\FF_q^2$.  
Now Scott's Lemma \cite{S} implies that
$\langle \alpha u,\alpha v,\alpha^{-2}w \rangle$  cannot act irreducibly 
on $\FF_q^2$ and so $\langle u,v,w \rangle \neq X$. See also \cite{GM}). }
 \end{example}  
 
\begin{example} \label{a7} 
{\em Let $G = \AAA_7$, the alternating group on $7$ letters. Using the character 
table of $G$ as given in \cite{Atlas}, one can check that there are elements  
$x,y,z \in G$ of order $2$, $5$, and $7$, respectively, with $xyz = 1$
and $G = \langle x, y \rangle$. 
However, if we denote by $\hat{x}$, $\hat{y}$, $\hat{z}$ the lifts of the same order 
of these elements in the central cover $X = 3\cdot \AAA_7$, then there are no 
$(u,v,w) \in \hat{x}^X \times \hat{y}^X \times \hat{z}^X$ such that $uvw = 1$, as 
one can check using the character table of $X$ in \cite{Atlas}.}
\end{example}

The previous two examples are given for $p = 2$ and $p = 3$. In \S5, using 
the classification of irreducible groups generated by pseudoreflections, we will give a family of 
examples for any odd prime $p$, cf. Example \ref{slp}. 

We can prove a version of Theorem \ref{main1} where we do allow
$p$ to divide the order of the Schur multiplier. Of course, this includes
Theorem \ref{main1} as a special case (since then $J=F$). 

\begin{theorem} \label{main1gen}    Let $X$ be a Frattini cover of 
$G = X/F$ with $F$ of odd order. Set $J=[F,X]$.
   Let $g_1, \ldots, g_r \in G$  satisfy

{\rm (i)} $G = \langle g_1, \ldots, g_r \rangle$,

{\rm (ii)}  $g_1 \cdots g_r =1$, and

{\rm (iii)} the order of each $g_i$ is coprime to $|F|$.\\ 
Let  $X_i := \{x_i \in X \mid x_iF = g_i, |x_i| = |g_i|\}$.
Then $X_1 \cdots X_r$ is a coset of $J$ in $F$.
\end{theorem}

\begin{proof}   First assume that $J=1$, i.e. $F \le Z(X)$.
It follows that there is a unique lift $x_i \in X$ of $g_i$ 
with $|x_i|=|g_i|$ and $x_iF=g_i$. Then $x_1 \ldots x_r = f \in F$.

Now go back to the general case. Then $F/J \leq Z(X/J)$ and so, as before, 
$g_i$ has a unique lift to $X/J$ of the same order. It follows   
that $X_1 \cdots X_r$ is contained in some coset $fJ$
of $J$ in $F$. It remains to show that each element in 
$fJ$ is a product of elements in $X_i$. We will induct on $|J|$. 
Suppose first that $J$ is central in $X$. Then
for any $h \in X$ of order $m$ coprime to $|F|$ and any $w \in F$, we have
$1 = [h^m, w]=[h,w]^m$, whence $[h,w] = 1$ and so 
$h$ centralizes $F$. In particular, $x_i \in X_i$ centralizes $F$ for 
$i = 1, \ldots, r$. But $\langle x_1, \ldots ,  x_r \rangle = X$ since 
$F \le \Phi(X)$ and $G = \langle g_1, \ldots, g_r \rangle$. It follows
that $F \leq Z(X)$, $J=1$, and so we are done by the previous case.

Now we may assume that $J$ is not central in $X$.
Let $N$ be a minimal normal noncentral subgroup of $X$ contained
in $J$; in particular, $N$ is a $p$-group for some prime $p$ dividing 
$|F|$ and so $p > 2$. By Proposition \ref{split}, $N$ is abelian.
%As in the proof of Theorem \ref{main1}, we see that 
%$p$ is coprime to $|X/[X,X]|$. 
Observe that, if $x_i \in X_i$ and $n \in N$, 
then $|n^{-1}x_in| = |x_i| = |g_i|$ and 
$n^{-1}x_inF = x_iF = g_i$, i.e. $n^{-1}x_in \in X_i$. 
By the induction hypothesis, the statement holds for $X/N$.
Applying Lemma \ref{abelian}, we see that the statement holds in $X$ as well.
\end{proof} 

The coset of $J$ in the previous result can be thought of as the lifting 
invariant (or obstruction).  As Fried has observed, the different lifting invariants give
rise to different orbits for the Hurwitz braid group acting on the corresponding
Nielsen classes of $G$, i.e. on 
$$\{(h_1, \ldots, h_r)  \mid \langle h_1, \ldots, h_r \rangle = G,
    ~h_1 \cdots h_r=1, ~h_i \in g_i^G \}.$$ 
See \cite{BF, intro} for more details.

\section{Solvable Groups}

We first need Thompson's result on minimal simple groups \cite{jgt}. 

\begin{theorem}\label{simple}   
Let $G$ be a finite simple group such that every proper subgroup
is solvable.   Then  $G$ is one of the following groups:

{\rm (a)}  $\PSL_{2}(p)$ with $p \ge 7$ an odd prime and $p \equiv \pm 2 (\mod 5)$;

{\rm (b)}  $\PSL_{2}(2^p)$  with $p$ a prime;

{\rm (c)}  $\PSL_{2}(3^p)$ with $p$ an odd prime;

{\rm (d)}  $\Sz(2^p)$ with $p$ an  odd prime; 

{\rm (e)} $\SL_{3}(3)$.
\end{theorem}

We need  one more preliminary result.

\begin{lemma}  \label{minimal}  
Let $G$ be a finite group that is not solvable 
but has the property that every proper subgroup is solvable.  
Then the solvable radical $R(G)$ of $G$ is $\Phi(G)$, the
Frattini subgroup of $G$; in particular, $R(G)$ is nilpotent. 
Moreover, $G/R(G)$ is a non-abelian simple group with
all proper subgroups being solvable, and every prime divisor of 
$|R(G)|$ divides $|G/R(G)|$. 
\end{lemma}

\begin{proof}  Suppose that $R := R(G)$ is not contained in $\Phi(G)$.
Then $G=RM$ for some maximal subgroup $M$ of $G$.
Then $M$ is solvable, whence $G/R$ is solvable and so is $G$.
Of course, $\Phi(G)$ is nilpotent, so $R=\Phi(G)$ is nilpotent.
Now $G/R$ has no nontrivial solvable normal subgroups; also, by the hypothesis,
it has no proper non-solvable subgroups. Hence $G/R$ is simple non-abelian. Finally,
suppose that a prime divisor $p$ of $|R|$ is coprime to $|G/R|$.
Since $R$ is nilpotent, $O_p(R) \cong R/O_{p'}(R)$. By the Schur-Zassenhaus
Theorem, $R/O_{p'}(R)$ has a complement $H/O_{p'}(R)$ in $G/O_{p'}(R)$. In this case,
$H$ is a proper non-solvable subgroup of $G$, a contradiction. 
\end{proof}

We next prove  Theorem \ref{main2} in the case that $G$ is quasi-simple. 

\begin{lemma}\label{qs}  
{\rm (i)} Let $G = \SL_2(q)$ with $q \geq 4$, respectively 
$G = \Sz(q)$ with $q \geq 8$, and let $p$ be any odd prime divisor of $|G|$.  
Then there exist nontrivial elements 
$x_i \in G$, $1 \le i \le 3$, such that $x_1x_2x_3=1$, 
$x_1$ is a $2$-element, $x_2$ is a $p$-element, and $x_3$ is an $s$-element 
for some prime $s \neq 2,p$.

{\rm (ii)} Let $G$ be a finite quasi-simple group with all proper
subgroups being solvable.
Then there exist nontrivial elements
$x_i \in G$, $1 \le i \le 3$, such that $x_1x_2x_3=1$, 
$x_1$ is a $2$-element, $x_2$ and $x_3$ are $p_i$-elements, where $p_2 < p_3$ 
are odd primes and $p_2 \in \{3,5\}$.
\end{lemma}

\begin{proof}  
(i) Recall that the number of 
triples $(y_1,y_2,y_3) \in G \times G \times G$ with $y_1y_2y_3 = 1$ and 
$y_i \in x_i^G$ is equal to 
\begin{equation}\label{count}
  \frac{|x_1^G| \cdot |x_2^G| \cdot |x_3^G|}{|G|^2} \cdot 
    \sum_{\chi \in \Irr(G)}\frac{\chi(x_1)\chi(x_2)\chi(x_3)}{\chi(1)}.
\end{equation} 
The character tables for $G = \SL(q)$ and $\Sz(q)$ are well known
(and are for example in {\sf Chevie} \cite{Ch}). 
First we consider the case $G = \Sz(q)$ with $q \geq 8$. Then choose 
$x_1$ of order $2$, $x_2$ of order $p$, and $x_3$ of some prime order $s$, where 
$s|(q-1)$ if $p|(q^2+1)$, and $s|(q^2+1)$ otherwise. 

Next suppose that $G = \SL_2(q)$ with $q = r^f \geq 4$, $r$ a prime.
 
$\bullet$ Assume $r= 2$ and let $\eps = \pm 1$ be chosen such that
$p|(q-\eps)$. Then we can choose $x_1$ of order $2$, $x_2$ of order $p$, and 
$x_3$ of some prime order $s$ dividing $q+\eps$. 

$\bullet$ Next assume that $p = r$. 
Choose $\eps = \pm 1$ such that $q \equiv \eps (\mod 4)$; in particular, 
$(q+\eps)/2$ has a prime divisor $s \neq 2,p$. Now we can choose 
$x_1$ of order $4$ (in a maximal torus $C_{q-\eps}$), $x_2$ of order $p$, and $x_3$ of 
order $s$ (in a maximal torus $C_{q+\eps}$). 

$\bullet$ Now let $r \neq 2,p$ but $q \not\equiv \pm 1 (\mod 4p)$. Then 
there is some $\eps = \pm 1$ such that $4|(q-\eps)$ and $p|(q+\eps)$, and   
we choose $x_1 \in C_{q-\eps}$ of order $4$, $x_2 \in C_{q+\eps}$ of order $p$, and 
$x_3$ of order $s = r$. 

In all the above cases, the principal character $1_G$ of $G$ is the only irreducible character 
that is nonzero at $x_1$, $x_2$, and $x_3$ altogether, and so we are done by (\ref{count}).

\smallskip
Suppose now that $G = \SL_2(q)$ with $q = r^f \geq 5$, 
$r \neq 2,p$ and $q \equiv \eps (\mod 4p)$ for some 
$\eps = \pm 1$. Choose $x_1 \in C_{q-\eps}$ of order $4$, $x_2 \in C_{q-\eps}$ of order 
$p$, and $x_3$ of some prime order $s$ dividing $(q+\eps)/2$ (so $s \neq 2,p$). 
Then there are 
precisely two irreducible characters of $G$ which are nonzero at all $x_i$: the 
principal character $1_G$, and the Steinberg characters $\St$ of degree 
$q$, with $|\chi(x_1)\chi(x_2)\chi(x_3)| = 1$. Hence we are done in this case as well.  

\medskip
(ii) It suffices to prove the statement for $G$, the Schur cover of $G/Z(G)$.
Hence we may assume that
$G = \SL_2(q)$, $\SL_3(3)$, $\Sz(q)$, or $2^2 \cdot \Sz(8)$ by Theorem 
\ref{simple}. We will choose $p = 3$ in the first two cases, and $p = 5$ in the 
last two cases. If $G = \SL_2(q)$ or $\Sz(q)$, then we are done by (i).
For $G = \SL_3(3)$, we can choose $x_1$ of order $2$, $x_2$ of order $3$, and 
$x_3$ of order $13$, and again $1_G$ is the only irreducible character that is 
nonzero at $x_1$, $x_2$, and $x_3$. For $G = 2^2 \cdot \Sz(8)$, we can choose 
$x_1$ a noncentral $2$-element, $x_2$ of order $5$, and 
$x_3$ of prime order $13$, and check our statement by using (\ref{count}) and 
\cite{Atlas}.
\end{proof}  

We can now prove a slightly stronger version of Theorem \ref{main2}. 

\begin{theorem} \label{main3}  
Let $G$ be a finite non-solvable group.
There exist nontrivial elements $x_i \in G$, $1 \le i \le 3$, such that $x_1x_2x_3=1$, 
$x_1$ is a $2$-element, $x_2$ and $x_3$ are $p_i$-elements, where 
$p_2 < p_3$ are odd primes and $p_2 \in \{3,5\}$.
\end{theorem}

\begin{proof}   Suppose the result is false.  Let $G$ be a minimal counterexample.
Thus, every proper subgroup of $G$ is solvable, and so $G$ is perfect and 
$G/\Phi(G)$ is non-abelian simple by Lemma \ref{minimal}.   

\smallskip
1) We first claim that $O_2(G)=1$.  If not, then the result holds for $G/O_2(G)$.
Thus, we can choose nontrivial  $x_i, 1 \le i \le 3$, where $x_i$ are $p_i$-elements 
with $p_1=2$, $p_2 <  p_3$ odd primes, $p_2 \in \{3,5\}$, and 
$x_1 x_2 x_3  = y \in O_2(G)$.
Then  $(y^{-1}x_1)x_2x_3=1$ and $y^{-1}x_1$ is a $2$-element (and is nontrivial
since $x_2$ and $x_3$ are). This proves the claim.

\smallskip
2) By Lemma \ref{qs}(ii), $\Phi(G)> Z(G)$. Let $N$ be a subgroup
of $\Phi(G)$ that is normal in $G$ but is not central.  Moreover, take
$N$ to be a minimal such subgroup. Since $N \leq \Phi(G)$ is nilpotent,  
the minimality of $N$ implies 
that $N$ is a $p$-group. Observe that $p > 2$ as $O_2(G) = 1$.  
By Proposition \ref{split}(ii), $N$ is abelian. Choose $x_i \in G$, $1 \le i \le 3$
with the $x_i$ nontrivial $p_i$-elements, $p_1 =2$,
$p_2 < p_3$ odd primes, $p_2 \in \{3,5\}$, and $x_1x_2x_3 = n \in N$ (this 
is possible since $G/N$ satisfies the theorem). Since 
%$G$ is perfect and 
every proper subgroup of $G$ is solvable, $G/N = \langle x_1N, x_2N,x_3N \rangle$
by Thompson's theorem \cite{jgt}. Furthermore, 
since $N \le \Phi(G)$, we see that $G = \langle x_1, x_2,x_3 \rangle$.
Recall that $G$ is perfect, and so $G = O^p(G)$. Hence by Lemma \ref{abelian},
there are conjugates $y_i$ of $x_i$ for $i = 1,2,3$ such that 
$y_1y_2y_3 = (x_1x_2x_3)n^{-1} = 1$,
a contradiction to the fact that $G$ was a counterexample. 
\end{proof}

The examples of $\Sz(8)$ and $\SL_3(2)$ show that Theorem \ref{main3}
is best possible, in the sense that one cannot always demand one of the primes 
to be $3$, respectively $5$. In \S\S4, 5 we will address possible refinements of 
Theorem \ref{main3} in some other directions.
  
\section{Prime Order Elements}

It is quite easy to see that Theorems \ref{main2} and  \ref{main3}  fail if we 
insist that the elements have 
prime order. For example, if $G \cong \SL_2(q)$ with $q$ odd and has order 
divisible by only $3$ primes, then since every involution in $G$ is central,  
no product of $3$ elements of distinct prime order (with one of the
primes equal to $2$) can be trivial (otherwise 
we would have two elements of prime order in $G/Z(G)$ with product being trivial).
For $q=5$ this was observed (with a more complicated proof) in \cite{barry}.   

More generally, we point out the following:

\begin{theorem}  \label{prime}  Let $G$ be a finite group. Then there exists a 
finite group $X$ with $X/F \cong G$ such that $F$ is abelian, 
$F \le \Phi(X)$, and $F$ contains all elements of prime order in $X$.
\end{theorem}

\begin{proof}  
Let $g \in G$ be of prime order $p$, and let $D$ be a one-dimensional
trivial $\langle g \rangle$-module over $\FF_p$.   
Since $\dim H^2(\langle g \rangle, D) =1$,
there exists a finite $\FF_pG$-module $W_g$ with an element 
$\beta_g \in H^2(G, W_g)$ 
such that the restriction of $\beta_g$ to $\langle g \rangle$ 
is nonzero in $H^2(\langle g \rangle, W_g)$.
(Just take the induced module for example).   
By taking direct sums, we see that 
there exists a finite $G$-module $W$ and an element $\beta \in H^2(G,W)$ such that
the restriction of $\beta$ to each subgroup of prime order in $G$  is nonzero. 
This allows us to construct an exact sequence
$$1  \rightarrow   W  \rightarrow E \rightarrow G \rightarrow 1,$$
where $W$ contains all elements of prime order in $E$  (since
$\beta|_C$ is nonzero in $H^2(C,W)$ for every subgroup of prime order $C$ of $G$).    

Now choose $X$ to be a minimal subgroup of
$E$ that surjects onto $G$.    Clearly $F := X \cap W$ is abelian and contains
all elements of prime order in $X$.   It only remains to show that 
$F \le \Phi(X)$. If not, then we could choose a proper subgroup $X_0$ of $X$ with 
$X = X_0F$ and so $X_0$ surjects onto $G$, contradicting the choice of $X$.
\end{proof}

In particular, given any non-solvable $G$, the subgroup generated by all elements 
of prime order in the extension $X$ specified in Theorem \ref{prime} is abelian, 
and so any product of elements in $X$ of distinct prime orders $p_i$ has order 
equal to the product of the primes.  

We close by noting that this example (especially for the prime $2$) has a nice
consequence for the inverse Galois problem.   As far as we know, this was first
observed by Serre.

\begin{cor}  
If every finite group occurs as a Galois group of a Galois extension $\KK$
of $\QQ$, then every finite group occurs as a Galois group over a totally real
Galois extension $\KK'$ of $\QQ$.
\end{cor}

\begin{proof}  
By Theorem \ref{prime}, there is a finite group $X$ such that $G = X/F$ and $F$ 
contains all involutions of $X$. Assume that $X = \mathrm{Gal}(\KK/\QQ)$
for some Galois extension $\KK$ of $\QQ$. Then $G =\mathrm{Gal}(\KK^F/\QQ)$.   
It remains to show that $\KK^F$ is totally real. Since $\KK^F/\QQ$ is Galois, it 
suffices to show that $\KK^F$ is real. Note that the complex conjugation
$\sigma$ acts on $\KK$.
%%%View $\KK$ as a splitting field of some polynomial $f$ over $\QQ$, so
%%%$\sigma$ preserves $f$ and so it acts on $\KK$.
Thus, $\sigma$ is an involution in $X$ and so contained  in $F$. Hence $\sigma$ is
trivial on $\KK^F$, whence $\KK^F$ is real.
\end{proof}  

\section{Finite Groups with $(2,p,q)$-triples}  \label{sec:psolvable}

First we recall a result essentially proved by Gow in \cite{Gow}:

\begin{lemma}\label{regular}
Let $G$ be a quasisimple Lie-type group of simply connected type, and let 
$x, y \in G$ be any two regular semisimple elements. Then $x^G \cdot y^G$ contains 
every noncentral semisimple element of $G$.
\end{lemma}

\begin{proof}
The proof is essentially given in \cite{Gow}. Let $p$ be the defining 
characteristic of $G$. Observe that $Z:= Z(G)$ is a $p'$-group,
so $PZ/Z \in \Syl_p(G/Z)$ when $P \in \Syl_p(G)$; also, 
$C_{G/Z}(PZ/Z)$ contains no nontrivial $p'$-element (cf. \cite{Gow}). Hence
$C_G(P)$ contains no noncentral semisimple element. It follows that if 
$g \in G \setminus Z(G)$ is semisimple, then $p$ divides $|g^G|$. Next, $G$ has 
exactly $s:=|Z|$ $p$-blocks of maximal defect $B_1, \ldots ,B_s$, and a single 
$p$-block of defect zero consisting of the Steinberg character $\St$. For any 
$i = 1, \ldots ,s$, there is an irreducible character $\chi_i \in B_i$ of 
$p'$-degree, and so the algebraic integer 
$\omega_{\chi_i}(g) := (\chi_i(g)/\chi_i(1))\cdot |g^G|$ is divisible by $p$. This implies
that $\omega_{\chi}(g) := (\chi(g)/\chi(1))\cdot |g^G|$ is divisible by $p$ for 
every irreducible character $\chi \in B_i$. On the other hand,
$|\St(x)| = |\St(y)| = 1$ since $x,y$ are regular semisimple. Thus 
$$\sum_{\chi \in \Irr(G)}\frac{\chi(x)\chi(y)\chi(g)}{\chi(1)} \cdot |g^G|
 \equiv \frac{\St(x)\St(y)\St(g)}{\St(1)} \cdot |g^G| 
 = \pm  [G:C_G(g)]_{p'} \not\equiv 0 (\mod p),$$
whence $g^{-1} \in x^G \cdot y^G$ by (\ref{count}).     
\end{proof}

\begin{example} \label{slp} 
{\em Let $p >  2$ be a prime, and $m > 1$ be coprime to $2p$.  
Assume moreover that $m$ does not divide $p-1$.   
Let $q$ be a prime with $mp|(q-1)$, and
let $G = SL_p(q)$. For brevity, we will outline the arguments only for $p > 7$. 

\smallskip
1) We choose three elements $x_i$ of $G$ as follows:

(a) $x_1$ is irreducible on a hyperplane (of the natural $\FF_qG$-module
$V = \FF_q^p$) and has an eigenvalue $1$;

(b) $x_2 = \diag(a ,b, \ldots, b)$ where $ab^{p-1}=1$ with $b \in \FF_q^{\times}$ 
of order $m$;
 
(c) $x_3$ is irreducible on a subspace of codimension $2$ of $V$, and has two  
eigenvalues, $1$ and $b^{-1}$, in $\FF_q$; furthermore it
has order coprime to $|x_1|$. 

Note that the orders of $x_i$, $i = 1,2,3$, are all coprime to $p$ and
$x_1$ and $x_3$ are regular semisimple elements.  It is easy to arrange
so that $|x_1|$ and $|x_3|$ are coprime (by taking $x_1$ to have
prime order $\ell$, a {\it primitive prime divisor} of $q^{p-1} -1$, cf. \cite{Zs}).
%; slightly more care is needed if $p=3$, but it can still be done easily). 

\smallskip
2) Observe that if $y_i$ is conjugate to $x_i$ with $y_1y_2y_3=1$, then
$\langle y_1, y_2 \rangle$ acts reducibly on $V$.  
This follows by Scott's Lemma \cite{S} applied to the elements
$y_1, b^{-1}y_2, by_3$.

\smallskip
3) Let $1 \ne d \in \FF_q^*$ be of order $p$. Then $d \notin\{1,a,b, ab^{-1}\}$ since
$b$, $a = b^{1-p}$, and $ab^{-1} = b^{-p}$ all have order dividing $m$. 
By Lemma \ref{regular} applied to the classes
$x_1^G$, $x_2^G$, $(d^{-1}x_3)^G$, there exist 
$u_i$ conjugate to $x_i$ such that $u_1u_2u_3 = dI$.
Let $U = \langle u_1, u_2, u_3 \rangle$. We claim that $U = G$ for any such 
choice of $u_i$.

First we see that $U$ acts irreducibly.  Since $u_1$ acts irreducibly on a
hyperplane, the only other possibility would be that $U$ preserves a hyperplane 
or a $1$-space. But then there would be a choice of eigenvalues $e_i$ of the $u_i$ 
with $e_1e_2e_3 = d$. Since $e_1=1$, $e_2 \in \{a,b\}$ and $e_3 \in \{1, b^{-1}\}$, 
$e_1e_2e_3$ is contained in $\{1,a,b, ab^{-1}\}$ and so cannot be equal to $d$.  

Now $U$ is an irreducible group containing pseudoreflections (up to scalar -- 
namely, $u_2$). Note also that, modulo scalars, $U$ is the normal closure of $u_2$ in 
$U$. (For if $N = \langle u_2^U \rangle$, then, modulo scalars, 
$U/N$ is generated by $u_1N$ and also 
by $u_3N$, and so has order dividing 
both $|u_1|$ and $|u_3|$, which are coprime). Thus, $U$ is generated by 
pseudoreflections of odd order modulo scalars. This implies that $U$ is primitive and
tensor-indecomposable, and this also excludes most of the ``obvious'' 
examples of pseudoreflection groups. Now the classification of finite 
pseudoreflection groups (cf. \cite[7.1]{gursaxl}) implies that $U = G$.   

\smallskip
4)  Let $\bar{u}_i \in \PSL_p(q)$ be the image of $u_i$.  Then 
$$\bar{u}_1\bar{u}_2\bar{u}_3 = 1,~~~ 
  \langle \bar{u}_1,\bar{u}_2,\bar{u}_3 \rangle = \PSL_p(q)$$
by the result of 3). But $\bar{u}_1$, $\bar{u}_2$, $\bar{u}_3$ do not lift to 
elements in $G$ of order coprime to $p$ with product equal to $1$ according to 2).}
\end{example}

Glauberman's classification of $S_4$-free simple groups \cite{Gl}, together with 
results of Goldschmidt \cite{Gol}, implies that Suzuki groups 
$\Sz(2^{2a+1})$, $a = 1,2, \ldots$, are the only finite non-abelian simple $3'$-groups.
We will need the full classification of finite simple groups to prove the 
following statement.

\begin{lemma}\label{3sol}
Let $S$ be a finite non-abelian simple group. Suppose that every proper  
subgroup of $S$ is $3$-solvable. Then $S$ is either a Suzuki group $\Sz(q)$, 
or one of the groups listed in Theorem \ref{simple}.
\end{lemma}

\begin{proof}
Certainly, $S \not\cong \AAA_{n}$ for any $n \geq 6$ since $\AAA_5$ is not 
$3$-solvable. Similarly, one can check using \cite{Atlas} that each of the 
$26$ sporadic simple groups has a section isomorphic to $\AAA_5$.

Next suppose that $S$ is a finite simple group of Lie type over $\FF_q$ with $q \geq 4$. 
If the twisted Lie rank of $S$ is at least $2$, then a proper section of $S$ is 
isomorphic to $\PSL_2(q)$ which is again not $3$-solvable. Otherwise $S \cong \PSL_2(q)$,
$\Sz(q)$,  $\PSU_3(q)$, or $\tw 2G_2(q)$. In the first case, it is easy to check that $S$ must 
then be one of the groups listed in Theorem \ref{simple}. The second case is included in 
the lemma's conclusion. The last two cases cannot happen as otherwise $S$ has a section 
$\cong \PSL_2(q)$. Also observe that $\SL_2(q^3)$ embeds in $\tw 3D_4(q)$.

Finally, suppose that $S$ is a finite simple group of Lie type over $\FF_q$ with $q = 2,3$,
not isomorphic to $\tw 3D_4(q)$, $\PSL_2(7)$, or $\SL_3(3)$. Considering Levi subgroups or
subsystem subgroups of $S$, one readily checks that $S$ has a section isomorphic to 
$\AAA_5$ or $\PSL_2(7)$.
\end{proof}

Let $p,q,r$ be primes. By a {\it $(p,q,r)$-triple} in a finite group $G$ 
we mean a triple $(x,y,z)$ of nontrivial elements in $G$ such that $x$ is a $p$-element,
$y$ is a $q$-element, $z$ is an $r$-element, and $xyz = 1$. 
(Note that in fact the order of $p,q,r$ does not matter, since if 
$xyz = 1$, then $y(x^{z^{-1}})z = 1$, etc.)
Theorem \ref{main3} then 
states that any finite non-solvable group admits either a $(2,3,p)$-triple 
(for a prime $p \geq 5$), or a $(2,5,p)$-triple for a prime $p \geq 7$. Now we can 
characterize finite groups $G$ for which only the latter can happen. In fact,
we can characterize the finite $p$-solvable groups (with $p > 2$) as precisely 
the ones which do not admit any $(2,p,q)$-triple for any prime $q \neq 2,p$. 

To this end, first we use the classification of finite simple groups to describe the
minimal non-$p$-solvable simple groups.

\begin{lemma}\label{psol1}
Let $S$ be a finite non-abelian simple group and $p \geq 5$ a prime with $p$
dividing $|G|$. 
If every proper subgroup of $S$ is $p$-solvable, then either $S = \PSL_2(p)$, 
$\AAA_p$, or one of the  following holds. 

{\rm (i)} $S = \PSL_2(q)$ with $p|(q^2-1)$.

{\rm (ii)} $S = \PSL_n(q)$, $n \geq 3$ is odd, and $p$ divides $q^n-1$ but not 
$\prod^{n-1}_{i=1}(q^i-1)$.

{\rm (iii)} $S = \PSU_n(q)$ $n \geq 3$ is odd, and $p$ divides $q^n-(-1)^n$ but 
not $\prod^{n-1}_{i=1}(q^i-(-1)^i)$.

{\rm (iv)} $S = \Sz(q)$.

{\rm (v)} $S = \tw2 G_2(q)$, and $p$ divides $q^2-q+1$ but not $q^2-1$.

{\rm (vi)} $S = \tw2 F_4(q)$, $q \geq 8$, and $p|(q^4-q^2+1)$.

{\rm (vii)} $S = \tw3 D_4(q)$, and $p$ divides $q^4-q^2+1$ but not $q^6-1$.

{\rm (viii)} $S = E_8(q)$, and $p$ divides $q^{30}-1$ but not 
$\prod_{i=8,14,18,20,24}(q^i-1)$.

{\rm (ix)} $(S,p)$ is one of: $(M_{23},23)$, $(J_1,7 \mbox{ or } 19)$, 
$(Ly,37 \mbox{ or } 67)$, $(J_4,29 \mbox{ or } 43)$, $(Fi'_{24},29)$, 
$(BM,47)$, $(M,41 \mbox{ or } 59 \mbox{ or } 71)$.  
\end{lemma}

\begin{proof}
By the assumption, every proper section of $S$ is $p$-solvable.
If $S = \AAA_n$, then $n = p$ since $\AAA_{p-1}$ is a $p'$-group and 
$\AAA_p$ is not $p$-solvable. The sporadic simple groups are treated using \cite{Atlas}.   

Suppose now that $S$ is a simple group of Lie-type in 
characteristic $r$. If $r = p \geq5$, then $S$ has a section isomorphic to 
$\PSL_2(p)$ which is not $p$-solvable, whence $S = \PSL_2(p)$. 
Now we consider the cases with $r \neq p$ type-by-type. 
The cases $S = \PSL_2(q)$ and $S = \Sz(q)$ are listed in (i) and (iv). 
Suppose $S = \PSL_n(q)$ with 
$n \geq 3$, but $p|(q^i-1)$ for some $2 \leq i \leq n-1$. 
In particular, if $n = 3$ then $q \geq 4$ since $p \geq 5$.
Hence, a proper simple section $\PSL_{n-1}(q)$ of $S$ is not $p$-solvable, a 
contradiction. Also, the case $p|(q^n-1)$ with $2|n$ is excluded 
by considering a section $\PSp_{n}(q)$ of $S$. Thus we arrive at (ii). 
Similarly, we arrive at (iii) if 
$S = \PSU_n(q)$ with $n \geq 3$. Next suppose that $S = \PSp_{2n}(q)$ with 
$n \geq 2$. If $p|(q^{2n}-1)$ then a proper section $\PSL_2(q^n)$ of $S$ 
is not $p$-solvable. Otherwise a proper section $\PSp_{2n-2}(q)$ of $S$ is not 
$p$-solvable. Similarly, if $S = \OO_{2n+1}(q)$ with $n \geq 3$, then 
either $\OO^{+}_{2n}(q)$ or $\OO^{-}_{2n}(q)$ is not $p$-solvable.
If $S = \OO^{+}_{2n}(q)$ with $n \geq 4$, then either 
$\OO^{+}_{2n-2}(q)$, or $\OO^{-}_{2n-2}(q)$, or $\PSL_n(q)$ is not $p$-solvable.
Suppose $S = \OO^{-}_{2n}(q)$ with $n \geq 4$. Then $S$ has two proper sections 
isomorphic to $\OO^{+}_{2n-2}(q)$ and $\OO^{-}_{2n-2}(q)$, and another 
proper section isomorphic to $\PSU_n(q)$ if $n$ is odd and 
$\PSL_2(q^n)$ if $n$ is even. At least one of these three sections is not 
$p$-solvable, a contradiction. If $G = \tw2 G_2(q)$, then $p$ cannot divide 
$q^2-1$ because of a section $\PSL_2(q)$, so we arrive at (v). 
If $G = G_2(q)$, then either $\PSL_3(q)$ or $\PSU_3(q)$ is not
$p$-solvable. If $G = \tw2 F_4(q)$, then by considering 
sections $\PSp_4(q)$ and $\PSU_3(q)$ of $S$ (and $\SL_3(3)$ for 
$q=2$) we see that $p|(q^4-q^2+1)$ and $q \geq 8$ as in (vi).
If $G = \tw3 D_4(q)$ but $p|(q^6-1)$, then a proper section 
$\PSL_2(q^3)$ is not $p$-solvable.
If $G = F_4(q)$, then a proper section
$\OO_9(q)$ or $\tw3 D_4(q)$ is not $p$-solvable. If $G = E_6(q)$, then a proper 
section $\OO^+_{10}(q)$, $\PSL_3(q^3)$, or $\tw3 D_4(q)$ is not $p$-solvable.
If $G = \tw2 E_6(q)$, then a proper 
section $\OO^-_{10}(q)$, $\PSU_3(q^3)$, or $\tw3 D_4(q)$ is not $p$-solvable.
If $G = E_7(q)$, then a proper section $E_{6}(q)$, $\tw2 E_6(q)$, or $\PSL_2(q^7)$ 
is not $p$-solvable. If $G = E_8(q)$ and $p|\prod_{i=8,14,18,20,24}(q^i-1)$, then a 
proper section $E_{7}(q)$, $\PSU_5(q^2)$, or $\tw3 D_4(q^2)$ is not $p$-solvable.
(The aforementioned sections of exceptional groups of Lie type come from
subgroups of maximal rank described in \cite{LSS}.)    
\end{proof}

\begin{lemma}\label{alt}
Let $q < p$ be odd primes. Then there exist $(2, q, p)$-triples in $\AAA_p$.
\end{lemma}

\begin{proof} 
Write $p = sq + t$ with $s \ge 1$ and $0 < t < q$.
Let $y \in \AAA_p$ be an element of order $q$ with $s$ nontrivial cycles.
Let $O_1, \ldots, O_s$ be the nontrivial orbits of $y$ and $O_0$ the 
set of all fixed points
of $y$ (acting on $\{1,2, \ldots,p\}$). Take an involution $x \in \SSS_p$ 
which is a product of $s+t-1$ transpositions so that 
$x$ moves exactly one point $a_{2i-1}$ of $O_i$ to $a_{2i} \in O_{i+1}$ for 
$i=1, \ldots, s-1$, and for each $j \in O_0$,  $xj$ is not in $O_0$ 
and disjoint from the points $a_1, \ldots ,a_{2s-2}$ (this is possible since 
$sq-2(s-1) \ge q \ge t$). For any $u \in \AAA_p$, let 
$\ind(u)$ be the difference between $p$ and the number of cycles of $u$,
i.e. the codimension of the fixed point subspace of 
$u$ on the natural permutation module $V = \CC^p$.
Then $\ind(x) = s + t-1$, and $\ind(y) = s(q-1)$. Note that
$s+t-1$ is even, so $x \in \AAA_p$. 
By construction $\langle x,y \rangle$ is a transitive subgroup
of $\AAA_p$. Applying Scott's Lemma \cite{S} to the action of 
$\langle x,y \rangle$ on $V$, we see that
$$\ind(x) + \ind(y) + \ind(xy) \ge 2p -2.$$
Hence, $\ind(xy) \ge p-1$, which forces $xy$ to be a $p$-cycle.
Thus, $(x, y, (xy)^{-1})$ is a $(2, q, p)$-triple.
\end{proof}

\begin{prop}\label{psol2}
Let $G$ be a finite quasisimple group and $p \geq 5$ a prime. 
Suppose that every proper subgroup of $G$ is $p$-solvable but $G$ is not 
$p$-solvable. Then $G$ admits a $(2,p,s)$-triple for some odd prime $s \neq p$.
\end{prop}

\begin{proof}
We apply Lemma \ref{psol1} to $S = G/Z(G)$ and may assume that $G$ is a Schur cover
of $S$. 

1) If $S = \PSL_2(p)$, then $G = \SL_2(p)$ since $p \geq 5$, 
and so we are done by Lemma \ref{qs}(i). Suppose $S = \AAA_p$. 
If $p \neq 7$, then by Lemma \ref{alt}, $S$ admits a $(2,p,q)$-triple, which
then lifts to a $(2,p,q)$-triple in $G = 2\AAA_p$ by p. 1) of the proof of 
Theorem \ref{main3}. Similarly, if $p = 7$, then by Lemma \ref{alt}, $S$ admits a 
$(2,3,7)$-triple, which first lifts to a $(2,3,7)$-triple in $2\AAA_7$ and then
lifts to a $(2,3,7)$-triple in $G = 6\AAA_7$. The same argument shows that 
$6 \cdot \PSL_2(9)$ admits a $(2,3,5)$-triple, and that
$2^2 \cdot \Sz(8)$ admits a $(2,p,s)$-triple for any $p \in \{5,7,13\}$. 
The cases $S = \PSL_2(q)$ and $\Sz(q)$ now follow from Lemma \ref{qs}(i). 

The sporadic groups listed in Lemma \ref{psol1} can certainly by checked 
using (\ref{count}) and \cite{Atlas}. But we point out a slightly easier way 
to check it as follows. Consider the case $S = M$ (so that 
$G = S$) and let $p \in \{41,59,71\}$.
Pick a prime $s \in \{41,59,71\} \setminus \{p\}$, $x \in S$ of order $p$, 
$y \in S$ of order $s$, and $z \in S$ of order $32$. Then 
$|C_G(x)| = p$, $|C_G(y)| = s$, $|C_G(z)| = 128$; also, if $\chi \in \Irr(G)$
is non-principal then $\chi(1) \geq 196883$. Since $|\Irr(G)| = 194$, it follows 
that 
$$|\sum_{\chi \in \Irr(G)}\frac{\chi(x)\chi(y)\bar{\chi}(z)}{\chi(1)}|
  \geq 1 - \frac{193 \cdot \sqrt{59 \cdot 71 \cdot 128}}{196883} > 0,$$
and so a $(2,p,s)$-triple exists. As another example, consider the case 
$S = Fi'_{24}$ (so we may assume $G = 3S$ and $p = 29$). Choosing $x \in G$ of 
order $p$, $y \in G$ of order $s=17$, and $z \in G$ of order $16$, we have  
$|C_G(x)| = 3p$, $|C_G(y)| = 3s$, $|C_G(z)| = 96$. Now if $\chi \in \Irr(G)$
is non-principal and $\chi(x)\chi(y)\bar{\chi}(z) \neq 0$, then 
$\chi(1) \geq 249548$. Since $|\Irr(G)| = 260$, it follows 
that 
$$|\sum_{\chi \in \Irr(G)}\frac{\chi(x)\chi(y)\bar{\chi}(z)}{\chi(1)}|
  \geq 1 - \frac{259 \cdot \sqrt{87 \cdot 51 \cdot 96}}{249548} > 0.$$
The same argument works for $S = M_{23}$ with $(p,s) = (23,11)$, 
$S = J_{1}$ with $(p,s) = (7,19)$, $S = Ly$ with $(p,s) = (37,67)$,
$S = J_{4}$ with $(p,s) = (29,43)$, and $S = BM$ with $(p,s) = (31,47)$.  

From now on we may assume that $S$ is a simple group of Lie-type in characteristic
$r \neq p$ (and not isomorphic to any of the aforementioned simple groups). 
The last assumption implies that $G$ is a Lie-type
group of simply connected type corresponding to $S$.

\smallskip
2) Assume in addition $r \neq 2$, so that we are in the cases (ii), (iii), (v),
(vii), or (viii) of Lemma \ref{psol1}.  In all these cases, the conditions on $p$
imply that a $p$-element $x \in G$ is regular semisimple (see e.g. 
\cite[Lemma 2.3]{MT} for exceptional groups). Clearly, Sylow 
subgroups of $G$ cannot be central (since $|S|$ and $|G|$ have same set of
prime divisors). Suppose that we can find a regular 
semisimple $s$-element $y \in G$, for a suitable prime $s \neq 2,p$. 
Then we can apply Lemma \ref{regular} to get a noncentral $2$-element 
$z \in x^G \cdot y^G$, yielding a $(2,p,s)$-triple. 

$\bullet$ Suppose $S = \PSL_n(q)$. If $n \geq 5$, or $n = 3$ but $q$ is not 
a Mersenne prime, then we can choose $s$ to be a primitive prime divisor 
of $q^{n-1}-1$. If $n = q = 3$, then $p = 13$ and we are done by Lemma \ref{qs}(i).  
In the remaining case, $G = \SL_3(q)$ and $q = 2^t-1 \geq 7$ is a 
Mersenne prime. Then $G$ contains a regular semisimple element $y$ of order 
$2^t$ (with eigenvalues $\alpha$, $\alpha^{-1}$, $1$, for some 
$\alpha \in \overline{\FF}_q^{\times}$ of order $2^t$). Also choose $s_1$ to be an 
odd prime divisor of $q-1 = 2^t-2$. By Lemma \ref{regular}, $x^G \cdot y^G$ 
contains a noncentral semisimple $s_1$-element $z$, giving rise to a 
$(2,p,s_1)$-triple.

$\bullet$ Suppose $S = \PSU_n(q)$. If $n \geq 5$, or $n = 3$ but $q$ is not 
a Fermat prime, then we can choose $s$ to be a primitive prime divisor 
of $(-q)^{n-1}-1$. If $n = q = 3$, then $p = 7$. In this case, 
picking $x$ of class $7A$, $y$ in class $2A$, and $z$ in class $3B$ in 
the notation of \cite{Atlas}, we see that $1_G$ is the only irreducible character
$\chi$ of $G$ such that $\chi(x)\chi(y)\bar{\chi}(z) \neq 0$, whence $G$ admits a 
$(2,3,7)$-triple. In the remaining case, $G = \SU_3(q)$ and $q = 2^{t}+1 \geq 5$ 
is a Fermat prime. Then $G$ contains a regular semisimple element $y$ of order 
$2^t$ (with eigenvalues $\alpha$, $\alpha^{-1}$, $1$, for some 
$\alpha \in \FF_q^{\times}$ of order $2^t$). Also choose $s_1$ to be an odd 
prime divisor of $q+1 = 2^t+2$. By Lemma \ref{regular}, $x^G \cdot y^G$ contains   
a noncentral semisimple $s_1$-element $z$, giving rise to a $(2,p,s_1)$-triple.

$\bullet$ Suppose $S = \tw2 G_2(q)$. Then there exist some $\eps = \pm 1$ such that 
$p|(q+\eps\sqrt{3q}+1)$, and some odd prime $s|(q-\eps\sqrt{3q}+1)$. Now
$G$ contains a regular semisimple $s$-element, and so we are done.

$\bullet$ Suppose $S = \tw3 D_4(q)$, and let $s$ be a primitive prime divisor 
of $q^3-1$. Then $G$ contains a regular semisimple 
$s$-element (of type $s_{12}$ as listed in \cite{DM}), and so we are done.

$\bullet$ Suppose $S = E_8(q)$, and let $s$ be a primitive prime divisor 
of $q^{24}-1$. By \cite[Lemma 2.3]{MT}, $G$ contains a regular semisimple 
$s$-element, and so we are done again.  

\smallskip
3) Assume now that $r = 2$, so that we are in the cases (ii),(iii), (vi)--(viii)
of Lemma \ref{psol1}.  In all these cases, the conditions on $p$
again imply that a $p$-element $x \in G$ is regular semisimple.

$\bullet$ Suppose $S = \PSL_n(q)$ or $S = \PSU_n(q)$. Set $\eps = 1$ in the $\SL$-case and
$\eps = -1$ in the $\SU$-case. If in addition $S \neq \SL_7(2)$, 
then we can choose an odd prime $s$ 
dividing $q^{n-1}-\eps^{n-1}$ but not $\prod^{n-2}_{i=1}(q^i-\eps^i)$. By the choice
of $s$, $G$ contains a regular semisimple $s$-element $y$, which is 
contained in a (unique) maximal torus of type $T_{1,n-1}$ of $G$, in the 
notation of \cite{LST}. The same is true for $S = \SL_7(2)$ if we choose 
$s = 3$ and $y \in S$ of order $9$.   
Also, $x$ is contained in a (unique) maximal torus of 
type $T_n$. Let $\chi \in \Irr(G)$ be any character of $G$ that is nonzero
at both $x$ and $y$. By Propositions 2.3.1 and 2.3.2 of \cite{LST}, the tori 
$T_n$ and $T_{1,n-1}$ are weakly orthogonal. Hence \cite[Proposition 2.2.2]{LST}
implies that $\chi \in \Irr(G)$  must be unipotent. Now the proofs of Theorems 
2.1 and 2.2 of \cite{MSW} 
imply that such a unipotent character $\chi$ is either trivial
or the Steinberg character $\St$. Since $x,y$ are regular, we also have 
$|\St(x)| = |\St(y)| = 1$. It follows by (\ref{count}) that 
$x^G \cdot y^G \supseteq G \setminus Z(G)$. In particular, 
we get a noncentral $2$-element $z \in x^G \cdot y^G$, yielding a 
$(2,p,s)$-triple. 

$\bullet$ Suppose $S = \tw3 D_4(q)$, and let $s$ be a primitive prime divisor 
of $q^3-1$. Then $G$ contains a regular semisimple $s$-element $y$ (of type 
$s_{12}$ as listed in \cite{DM}). Let $T_1$, respectively $T_2$, be the 
unique maximal torus containing $x$, respectively $y$. If $\Phi_m(q)$ denotes 
the $m^{\mathrm {th}}$ cyclotomic polynomial in $q$, then 
$|T_1| = \Phi_{12}(q)$ and $|T_2| = \Phi_3(q)^2$. The order of the centralizer of any
semisimple element in the dual group $G^* \cong G$ is listed in 
\cite[Tables 1.1, 2.2]{DM}. Using this, it is easy to see that the centralizer of 
no nontrivial semisimple element of $G^*$ can have order divisible by both 
$|T_1|$ and $|T_2|$. Thus the dual tori $T_1^*$ and $T_2^*$ in $G^*$ intersect
trivially, and so $T_1$ and $T_2$ are weakly orthogonal. By 
\cite[Proposition 2.2.2]{LST}, any irreducible character $\chi \in \Irr(G)$ 
that is nonzero on both $x$ and $y$ must be unipotent. Note that 
the $p$-parts of $\Phi_{12}(q)$ and of $|G|$ are the same. Hence, if $\Phi_{12}(q)$
divides $\chi(1)$, then $\chi$ has $p$-defect $0$ and so $\chi(x) = 0$. 
Similarly, if $\Phi_{3}(q)^2$ divides $\chi(1)$, then $\chi$ has $s$-defect $0$ and
so $\chi(y) = 0$. Inspecting the list of unipotent characters as given in \cite{Sp},
we see that $\chi = 1_G$, $\St$, or the unique unipotent character $\rho$ of 
degree $q^3(q^3+1)^2/2$. Choosing $z$ to be a unipotent element of class $D_4(a_1)$
of \cite{Sp}, we see that $\rho(z) = \St(z) = 0$. It now follows
by (\ref{count}) that $z \in x^G \cdot y^G$, giving rise to a 
$(2,p,s)$-triple.

$\bullet$ Suppose $S = \tw2 F_4(q)$ with $q > 2$. Then there exist some 
$\eps = \pm 1$ such that $p|(q^2+q+1+\eps\sqrt{2q}(q+1))$, and some 
primitive prime divisor $s$ of $q^6-1$, and $G$ contains a regular semisimple 
$s$-element by \cite[Lemma 2.3]{MT}. In particular,
$$|C_G(x)| = q^2+q+1+\eps\sqrt{2q}(q+1),~~~|C_G(y)| = q^2-q+1.$$
Next, take $z$ to be a regular unipotent element, so that 
$|C_G(z)| \leq 4q^2$ \cite{LiS}. In particular,
$$|\chi(x)\chi(y)\bar{\chi}(z)| < \sqrt{|C_G(x)| \cdot |C_G(y)| \cdot |C_G(z)|} <
  (2.7)q^3$$
for any $\chi \in \Irr(G)$. Now $G$ has 
exactly two irreducible characters of degree $(q^2-1)(q^3+1)\sqrt{q/2}$, and 
all other nontrivial irreducible characters have degree at least 
$q(q^2-q+1)(q^4-q^2+1)$, cf. \cite{Lu}. Also, 
$|\Irr(G)| \leq q^2+4q+17 < (1.8)q^2$ \cite[Table 1]{FG}. 
It now follows by (\ref{count}) that 
$$|\sum_{\chi \in \Irr(G)}\frac{\chi(x)\chi(y)\bar{\chi}(z)}{\chi(1)}| 
  > 1 - \frac{(1.8)q^2 \cdot (2.7)q^{3}}{q(q^2-q+1)(q^4-q^2+1)} - 
    \frac{(2.7)q^{3}}{(q^2-1)(q^3+1)\sqrt{q/2}} > 0,$$
whence $z \in x^G \cdot y^G$, giving rise to a $(2,p,s)$-triple.  

$\bullet$ Suppose $S = E_8(q)$, and let $s$ be a primitive prime divisor 
of $q^{24}-1$. By \cite[Lemma 2.3]{MT}, $G$ contains a regular semisimple 
$s$-element $y$. Then 
$|C_G(x)| = \Phi_m(q)$ with $m \in \{15,30\}$, and $|C_G(y)| = \Phi_{24}(q)$. 
Next we choose $z$ to be a regular unipotent element, so that 
$|C_G(z)| \leq 4q^8$ \cite{LiS}. Now for any nontrivial $\chi \in \Irr(G)$ we have 
that $\chi(1) > q^{27}(q^2-1)$ by the Landazuri-Seitz-Zalesskii bound \cite{LS},
and 
$$|\chi(x)\chi(y)\bar{\chi}(z)| < \sqrt{|C_G(x)| \cdot |C_G(y)| \cdot |C_G(z)|} <
  q^{14}.$$
On the other hand, $|\Irr(G)| < (5.1)q^8$ \cite[Table 1]{FG}. 
It now follows by (\ref{count}) that 
$$|\sum_{\chi \in \Irr(G)}\frac{\chi(x)\chi(y)\bar{\chi}(z)}{\chi(1)}| 
  > 1 - \frac{(5.1)q^8 \cdot q^{14}}{q^{27}(q^2-1)} > 0,$$
whence $z \in x^G \cdot y^G$, giving rise to a $(2,p,s)$-triple.  
\end{proof}

We can now prove Theorem \ref{main4}.  We restate the result.

\begin{theorem}
Let $G$ be a finite group. Then the following statements are equivalent:

{\rm (i)} $G$ admits no $(2,p,q)$-triple for any odd prime $q \neq p$;

{\rm (ii)} $G$ is $p$-solvable. 

{\rm (iii)} If $T$ is a Sylow $2$-subgroup of $G$, $P$ a Sylow $p$-subgroup
of $G$ and $Q$ a Sylow $q$-subgroup of $G$ (for any prime $q \ne p$), then
$TP \cap Q = 1$. 
\end{theorem}

\begin{proof}
1) Suppose that $G$ is $p$-solvable but admits a $(2,p,q)$-triple for some prime 
$q \neq 2,p$. Choose such a $G$ of minimal order. Since $p$ divides $|G|$ and $G$ is 
$p$-solvable, $G$ cannot be simple. Let $1 < N < G$ be a normal subgroup of $G$. 
If $x,y,z$ all belong to $N$, then $(x,y,z)$ is a $(2,p,q)$-triple in $N$, and 
so by minimality of $G$, $N$ cannot be $p$-solvable, contradicting the 
$p$-solvability of $G$. So at least one of $x,y,z$ is not contained in $N$. In this case,
since $xyz=1$, all of them are outside of $N$ by order consideration. 
It follows that $(xN,yN,zN)$ is a $(2,p,q)$-triple in $G/N$, and so $G/N$ is not 
$p$-solvable by minimality, again a contradiction.

\medskip
2) From now on we will assume that $G$ is a not $p$-solvable and aim to show that $G$ admits 
a $(2,p,q)$-triple for some $q \neq 2,p$. Consider a minimal counterexample $G$ to this 
claim, so that $G$ has a composition factor $S$ which is not $p$-solvable but $G$ admits
no $(2,p,q)$-triple with $q \neq 2,p$. The minimality of $G$ implies that any 
proper subgroup of $G$ is $p$-solvable and that $G$ is perfect. 

Let $N$ be any proper normal subgroup of $G$, in particular, $N$ is $p$-solvable.
Suppose in addition that $N$ is not contained in $\Phi(G)$. Then $G=MN$ for 
a maximal subgroup $M < G$. Then $S$ is also a composition factor of $M$, whence 
$M$ is not $p$-solvable, a contradiction. Thus every proper normal subgroup of $G$
is contained in $\Phi(G)$. It follows that $G/\Phi(G)$ is simple and so it is isomorphic to
$S$. Since every proper subgroup of $G$ is $p$-solvable, the same holds for $S$,
whence $S$ is one of the groups listed in Lemma \ref{3sol} if $p = 3$ 
and in Lemma \ref{psol1} if $p \geq 5$.

\medskip
3) Suppose $\Phi(G) \leq Z(G)$. Then $G$ is a quasisimple group with all proper 
subgroups being $p$-solvable.

Assume $p = 3$. Then $S \not\cong \Sz(q)$ as $S$ is not $3$-solvable, so $S$ 
is one of the groups listed in Theorem \ref{simple}. 
This in turn implies that all proper subgroups of 
$S$ are solvable. This is also true for any proper subgroup $H < G$. (Indeed, 
in this case $H\Phi(G) < G$, whence $H\Phi(G)/\Phi(G) < S$ is solvable and 
$\Phi(G)$ is nilpotent.) Now Lemma \ref{qs} (and its proof) implies that 
$G$ admits a $(2,3,q)$-triple for some prime $q \neq 2,3$, a contradiction.
   
On the other hand, if $p \geq 5$, then according to Proposition \ref{psol2}, 
$G$ also admits a $(2,p,q)$-triple for some $q \neq 2,p$, again a contradiction. 

We have shown that $\Phi(G) > Z(G)$. 
Furthermore, if $O_2(G) \neq 1$, then $G/O_{2}(G)$ admits a $(2,p,q)$-triple 
for some prime $q \neq 2,p$ by minimality of $G$, which can then be lifted to
a $(2,p,q)$-triple in $G$ (see p.1) of the proof of Theorem \ref{main3}). 
Thus $O_2(G) = 1$. 

Let $N$ be a subgroup of $\Phi(G)$ that is normal in $G$ but is not central.  
Moreover, take $N$ to be a minimal such subgroup. Then the minimality implies
that $N$ is an $r$-group for some prime $r$, and $r > 2$ since $O_2(G) = 1$. 
This in turn implies by Proposition \ref{split}(ii) that $N$ is abelian. 
By the minimality of $G$, there are some nontrivial elements $x_i \in G$, $1 \le i \le 3$,
such that $x_1$ is a $2$-element, $x_2$ is a $p$-element, $x_3$ is a $q$-element
for some prime $q \neq 2,p$, and $x_1x_2x_3 = n \in N$. In particular,
the subgroup $L/N = \langle x_1N,x_2N,x_3N \rangle$ of $G/N$ admits a $(2,p,q)$-triple.
According to 1), $L/N$, and so $L$, is not $p$-solvable. Hence $L = G$. This implies
that $G = \langle x_1,x_2,x_3 \rangle$ since $N \leq \Phi(G)$.   
Now arguing as in p. 3) of the
proof of Theorem \ref{main3} and using  Lemma \ref{abelian}, we see that 
there are conjugates $y_i$ of $x_i$ for $i = 1,2,3$ such that 
$y_1y_2y_3 = (x_1x_2x_3)n^{-1} = 1$,
a contradiction to the fact that $G$ was a counterexample. 

The equivalence of (i) and (iii) is straightforward.
\end{proof}  

Theorem \ref{main4} yields the following immediate consequence:

\begin{cor}
Let $G$ be a finite group. Then the following statements are equivalent:

{\rm (i)} $G$ admits no $(2,3,p)$-triple for any prime $p \geq 5$;

{\rm (ii)} $G$ is $3$-solvable;

{\rm (iii)} Every composition factor of $G$ is either cyclic or a Suzuki group. 
\end{cor}

\section{Examples with $p=2$} \label{even}

One of the key steps in proving Theorem \ref{main1} was
Proposition \ref{split}.  As we have observed, Proposition \ref{split} fails for 
$p=2$.  Here we produce examples showing that in fact
Theorem \ref{main1} fails for $p=2$ as well.

Let $E = 2^{1+2n}_{-}$ be the extraspecial group of type $-$ of order $2^{1 + 2n}$
for any $n \geq 5$. It is
well known that there is a non-split extension $G$ of $E$ such that 
$G/E \cong H:=\OO^-_{2n}(2)$.
Then $G$ has a a complex irreducible character $\varphi$ of degree
$2^n$ which is irreducible and faithful when restricted to $E$.
For $x \in G$, let $\bar{x} := xE$ be the corresponding element of $H$.

Let $x_i \in G$ of order $2^{n_i} + 1$ be acting on $E/Z(E)$ with
one nontrivial irreducible submodule of dimension $2n_i$
(and trivial on a complement). 
It follows by \cite[p. 372]{Gor}  that:
\begin{equation}\label{value}
  \varphi(x_i) = - 2^{n-n_i}.
\end{equation}

Now we show that Theorem \ref{main1} fails for $p = 2$.

\begin{prop}\label{p=2}  
In the above notation, choose $n_1=1$, $n_2 = n-1$, and $n_3=n$.

{\rm (i)} If $y_i \in G$ is such that $\langle y_i \rangle$ is conjugate to 
$\langle x_i \rangle$ for $i =1,2,3$, then $y_1 y_2 y_3 \ne 1$.

{\rm (ii)} There are conjugates $z_i \in G$ of $x_i$ such that 
$$\langle \bar{z}_1, \bar{z}_2, \bar{z}_3 \rangle = H = \OO^-_{2n}(2),~~~ 
  \bar{z}_1\bar{z}_2\bar{z}_3 =1,$$
but the generating triple $(\bar{z}_1,\bar{z}_2,\bar{z}_3)$ of 
$H$ does not lift to any triple $(t_1,t_2,t_3)$ in $G$ with 
$\bar{t}_i = \bar{z}_i$, $|t_i| = |z_i|$, and $t_1t_2t_3=1$. 
\end{prop}

\begin{proof}  
(i) Note that any generator of $\langle x_i \rangle$ also fulfills the conditions 
imposed on $x_i$. Hence we may assume that $y_i \in x_i^G$.
We use (\ref{count}) to count the number $N$ of triples  
in $x_1^G \times x_2^G \times x_3^G$ with product $1$,
%\begin{equation}
% N= \frac{|x_1^G| \cdot |x_2^G| \cdot |x_3^G|}{|G|^2} \cdot 
%    \sum_{\chi \in \Irr(G)}\frac{\chi(x_1)\chi(x_2)\chi(x_3)}{\chi(1)}.
%\end{equation} 
and break the sum in (\ref{count}) into three pieces.  The first piece is the
sum over the irreducible characters whose kernel contains
$E$. The second piece is the sum over the characters whose
kernel is $Z(E)$ and the third piece is the sum over all faithful characters.
Let $N_1$ denote the first sum and let $N_3$ denote the third sum.

First we note that any character $\beta$ in the second sum is afforded
by an induced module from the stabilizer of a linear character of $E/Z(E)$. Since $x_3$
has no fixed points on $E/Z(E)$, any such character $\beta$ vanishes
on $x_3$. Thus the second sum is $0$. 

Next, by Gallagher's theorem \cite[6.17]{I2}, the characters in the third sum are 
precisely those of the form $\varphi\lambda$ where $\lambda$ is an irreducible
character of $G/E$ 
%\cite[51.7]{CR}  
(and they are all distinct for distinct $\lambda$).   
Applying (\ref{value}), we now see that
$$N_3 = \frac{\varphi(x_1)\varphi(x_2)\varphi(x_3)}{\varphi(1)} N_1 = - N_1,$$
whence $N=0$.

(ii) Note that $\bar{x}_2$ and $\bar{x}_3$ are regular semisimple elements of
$H$. Hence, by Lemma \ref{regular}, there exist $z_i \in G$ 
conjugate to $x_i$ for $i = 1,2,3$ such that $\bar{z}_1\bar{z}_2\bar{z}_3 =1$. 
By \cite[7.1]{gursaxl}, the elements $\bar{z}_i$ generate $H$.
Now consider any $t_i \in z_iE = \bar{z}_i$ with $|t_i| = |z_i|$. By the 
Schur-Zassenhaus theorem, $\langle t_i \rangle$ is conjugate to 
$\langle z_i \rangle$ and hence to $\langle x_i \rangle$. It follows
by (i) that $t_1t_2t_3 \neq 1$. 
\end{proof}

Note that one can construct similar examples for odd $p$ with $E$
extraspecial of exponent $p$ of order $p^{1+2a}$ and $G/E \cong \Sp_{2a}(p)$.  
However, in this case the extension is split and so $G$ is not a Frattini cover.  

\smallskip
For the next example, 
let $E$ and $G$ as in the beginning of the section with  $n= 2m \geq 6$.
Also we choose $m$ such that $2^{2m}-1$ has at least two different 
primitive prime divisors $p_1$ and $p_2$. This is possible for instance
for $m = 14$, with $p_1 = 29$ and $p_2 = 113$. Also fix a primitive 
prime divisor $p_3$ of $2^{4m}-1$, and choose $n_1 = n_2 = m$
and $n_3 = n$. Since $E/Z(E)$ is a quadratic space of type $-$, one can
check that, for $i = 1,2,3$, any nontrivial $p_i$-element $x_i \in G$ acts irreducibly 
on a subspace of dimension $2n_i$ of $E/Z(E)$ and trivially on a complement.
%but of prime power order $p_i^{e_i}$  with the primes $p_i$ distinct. 
Arguing exactly as above, we see that there are no conjugates $y_i$
of $x_i$ in $G$ with $y_1y_2y_3=1$.  
%Moreover, since the
%Sylow $p_i$-subgroups of $G$ are cyclic, we see that for
%$y_i$ a $p_i$-element for $i=1,2,3$, we have $y_1y_2y_3 \ne 1$.
Thus, there  is no $(p_1, p_2, p_3)$-triple in $G$, but $G$ has a
composition factor whose order is divisible by $p_1p_2p_3$.

\smallskip
We know of no such counterexample with one of the primes being even.
We conjecture:

\begin{conj}\label{2pq}  
Let $q < p$ be odd primes and let $G$ be a finite group.  
The following statements are equivalent:

{\rm (i)} $G$ contains a composition factor whose order is divisible by $pq$; and

{\rm (ii)} $G$ contains a $(2,p,q)$-triple.
\end{conj}

The same arguments as in the proof of Theorem 
\ref{main4} show that (ii) implies (i), and that
a minimal counterexample to Conjecture \ref{2pq} would be a quasisimple group $G$ 
such that $G$ has no simple sections of order divisible by $pq$.  Moreover,
we can assume that $O_2(G)=O_p(G)=O_q(G)=1$. Note
that the result holds for $G=\AAA_n$ by Lemma \ref{alt} and so for the covering
groups as well (checking $6\AAA_7$ directly).

\section{A Short Proof of the Feit-Tits Theorem}
Let $\FF$ be an algebraically closed field of characteristic $p \ge 0$.
If $S$ is a finite non-abelian simple group, let $m_p(S)$ be the smallest positive integer
$n$ such that $S$ is a section of some subgroup of $\GL_n(\FF)$.
Also, let $d_p(S)$ be the smallest degree of a nontrivial representation
of the covering group of $S$ over $\FF$ (i.e. the smallest nontrivial degree of a
projective representation of $S$ over $\FF$). The following theorem was 
proved by Feit and Tits in \cite{FT} (and it was refined further by 
Kleidman and Liebeck \cite{KL} using the classification of finite simple 
groups). Here we give 
a short proof of the Feit-Tits theorem.

\begin{theorem}   
Suppose  $m_p(S) \ne d_p(S)$ for a finite non-abelian simple group $S$. 
Then $p \ne 2$ and $m_p(S) = 2^{n(S)}$, where $n(S)$ is the smallest positive 
integer $n$ such that $S$ embeds in $\Sp_{2n}(2)$. 
%Moreover, $n(S) \geq d_2(S)$.
\end{theorem}

\begin{proof}
1) Certainly, $m_p(S) \leq d_p(S)$. Also, suppose that $p \neq 2$ and $S$ embeds in some 
$\Sp_{2n}(2)$. Since $\GL_{2^n}(\FF)$ contains a subgroup of the form 
$(C_4 \circ 2^{1+2n}_{+}) \cdot \Sp_{2n}(2)$, we see that $m_p(S) \leq 2^n$. Thus 
$m_p(S) \leq 2^{n(S)}$.  

Set $m:=m_p(S)$ and let $H \le GL_m(\FF)= \GL(V)$ where $H$ is a finite
group with $S$ a section of $H$. By passing to a subgroup, we may
assume that $H$ surjects onto $S$ and no proper subgroup of $H$  
surjects onto $S$. Now, for any proper normal subgroup $N$ of $H$, if $N$ is
not contained in a maximal subgroup $M$ of $G$, then $MN = H$ and so 
$M/(M \cap N) \cong H/N$. By the minimality of $H$, $S$ is not a composition factor of 
$N$, hence it is a composition factor of $H/N$. It follows that a subgroup of $M$ 
projects onto $S$, a contradiction. Thus every proper normal subgroup of 
$H$ is contained in $\Phi(H)$, and so $H/\Phi(H) \cong S$.
It also follows that $H$ is perfect.
Since $m_p(S) < d_p(S)$ by the hypothesis, we see that $\Phi(H) > Z(H)$. 

\medskip
2) Clearly, we may assume that $H$ is an irreducible subgroup of $\GL(V)$.
Now we choose $N \leq \Phi(H)$ to be a noncentral normal subgroup of $H$ of smallest
possible order.  So $N$ is an $r$-group for some prime $r$. If $r=p$, then
$N$ acts trivially on $V$ by irreducibility, a contradiction (since $H \le GL(V)$).
Thus $r \neq p$.

If $N$ is abelian, then, since $N$ is not central, $m$ is at least the size of the 
smallest orbit of $H$ on the set of nontrivial irreducible characters of $N$ 
by Clifford's theorem.  This implies that some
nontrivial homomorphic image $H/K$ of $H$ embeds in ${\sf S}_m$. 
By 1), $S$ is a quotient of
$H/K$, and so $m = m_p(S) \le m_p({\sf S}_m) \le m-1$, a contradiction.
On the other hand, if $r \neq 2$ then $N$ must be abelian by Proposition 
\ref{split}(ii). Hence 
$p \neq r = 2$. In particular, we see that $O_{2'}(\Phi(H)) \le Z(H)$ and the
only noncentral normal subgroups of $H$ contained in $\Phi(H)$ are non-abelian $2$-groups. 

Let $M$ be any normal subgroup of $H$ properly contained in $N$. Then $M$ is central by 
the minimality of $N$, and so $M \leq Z(H)$. In particular, 
$Z(N)= N \cap Z(H)$ and so it is cyclic by Schur's lemma. 
Furthermore, $Z(N/Z(N)) = N/Z(N)$, and $N/Z(N)$ is an elementary abelian 
$2$-group. Also, $H$ acts irreducibly on $N/Z(N)$, and the commutator
map defines an $H$-invariant nondegenerate alternating form on $N/Z(N)$ (with values in 
$\Omega_1(Z(N)) \cong C_2$). Since $O_{2'}(\Phi(H)) \leq Z(H)$ acts trivially on
$N/Z(N)$, the irreducibility of $H$ on $N/Z(N)$ implies that $\Phi(H)$ acts 
trivially on $N/Z(N)$. Recall that $S = H/\Phi(H)$ is simple. Now if $S$ acts trivially
on $N/Z(N)$, then $[H,N] \leq Z(N)$ centralizes $H$ and so $[H,N] = [[H,H],N] = 1$
by the Three Subgroups Lemma, a contradiction as $N$ is not central. Thus $S$ 
acts faithfully on $N/Z(N)$ and so embeds in $\Sp_{2n}(2)$, where $|N/Z(N)|= 2^{2n}$.
In particular, $n \geq n(S)$.

Let $m\varphi$ denote the character of the $Z(N)$-module $V$ (recall $Z(N)$ acts 
scalarly and faithfully on $V$). Then the restriction of $\varphi$ to 
$1 \neq [N,N] \leq Z(N)$ is nontrivial. By Proposition \ref{split}(i),
there is exactly one irreducible character $\theta$ of $N$ lying above $\varphi$,
and $\theta(1) = 2^n$; also $\theta$ is $H$-invariant. Thus the $N$-module $V$ 
is a direct sum of $m/2^n$ copies of a single irreducible $\FF N$-module which 
affords the character $\theta$. Let $\Psi$ denote the representation of $H$ on $V$. 
By Clifford theory, $\Psi \cong \Theta \otimes \Lambda$, where $\Theta$ 
is an irreducible projective representation of $H$ of dimension $2^n$ and $\Lambda$ 
is an irreducible projective representation of $H/N$. Now if $\Lambda$ is nontrivial, then 
$m = \dim(V)  > \deg(\Lambda) \ge m_p(S)$, a contradiction. Thus $m=2^n$.
Since $n \geq n(S)$ and $m \leq 2^{n(S)}$, we conclude that $m = 2^{n(S)}$. 
%As we have already noted, $d_2(S)= m_p(S)$, whence the last statement holds.
\end{proof}

\end{document}